\documentclass[11pt]{article}
\usepackage{amssymb}
\usepackage{amsthm}
\usepackage{amsmath}
\usepackage{graphicx}
\usepackage{empheq}
\usepackage{indentfirst}
 \newtheorem{thm}{Theorem}[section]
 \newtheorem{cor}{Corollary}[section]

 \theoremstyle{definition}
 
 \newtheorem{rem}{Remark}[section]
 \numberwithin{equation}{section}
\def\f{\frac}
\def\pa{\partial}
\def\pd #1#2{\f{\pa #1}{\pa #2}}

\def\e{\eqref}
\def\la{\lambda}

\def\vph{\varphi}
\def\wt #1{\widetilde{#1}}
\def\i1n{(i=1,\cdots,n)}
\def\j1n{(j=1,\cdots,n)}
\def\ij1n{(i,j=1,\cdots,n)}

\newcommand{\be}{\begin{equation}}
\newcommand{\ee}{\end{equation}}

\begin{document}

%-----------------------------------------------------------------------------------------------------------title
\begin{center}

{\LARGE\bf Exact boundary controllability and observability \\[2mm] for
first order quasilinear hyperbolic systems with a \\[5mm] kind of nonlocal boundary conditions}\\[5mm]

{\large Tatsien Li \footnote{School of Mathematical Sciences, Fudan
University, Shanghai 200433, China. Shanghai Key Laboratory for
Contemporary Applied Mathematics. \ Email: {\tt dqli@fudan.edu.cn}.
Tatsien Li was partially supported by the National Basic Research
Program of China (973 Program, Grant No. 2007CB814800).}
\quad Bopeng Rao\footnote{Institut de Recherche Math\'{e}matique
Avanc\'{e}e, Universit\'{e}  de Strasbourg, 67084 Strasbourg,
France. \  Email: {\tt rao@math.u-strasbg.fr}.}
\quad Zhiqiang Wang\footnote{School of Mathematical Sciences, Fudan
University, Shanghai 200433, China. Universit\'{e} Pierre et Marie
Curie-Paris 6, UMR 7598 Laboratoire Jacques-Louis Lions, 75005
Paris, France. E-mail: {\tt wzq@fudan.edu.cn}. Zhiqiang Wang was
supported by the Natural Science Foundation of China (Grant No.
10701028) and Fondation Sciences Math\'{e}matiques de Paris.}}

\end{center}

\vspace*{5mm}
%\maketitle

\begin{abstract}
In this paper we  establish the theory on semiglobal classical
solution to first order quasilinear hyperbolic systems with a kind
of nonlocal boundary conditions, and based on this, the
corresponding exact boundary controllability and observability are
obtained by a constructive method. Moreover, with the linearized
Saint-Venant system and the 1-D linear wave equation as examples, we
show that the number of both boundary controls and boundary
observations can not be reduced, and consequently, we conclude that
the exact boundary controllability for a hyperbolic system in a
network with loop can not be realized generically.

\end{abstract}


{\bf 2000 MR Subject Classification:}\quad
35L50,
%35L50 (1973-now) Boundary value problems for hyperbolic systems of first-order PDE
93B05,
%93B05 (1973-now) Controllability
93B07
%93B07 Observability

\section{Introduction}%-------------------------------------------------------------------introduction

Consider the following first order quasilinear hyperbolic system
\be\label{sys-original} \pd ut+A(u)\pd ux=B(u), \ee
where $u=(u_1,\cdots,u_n)^T$ is the unknown vector function of
$(t,x)$, $A(u)$ is a $n\times n$ matrix with suitably smooth entries
$a_{ij}(u)\ \ij1n$, $B(u)=(b_1(u),\cdots,b_n(u))^T$ is a suitably
smooth vector function with
\be\label{1.2} B(0)=0. \ee

By hyperbolicity, for any given $u$ on the domain under
consideration, the matrix $A(u)$ possesses $n$ real eigenvalues
$\lambda_1(u),\cdots, \lambda_n(u)$ and a complete set of left
eigenvectors $l_i(u)=(l_{i1}(u),\cdots,l_{in}(u))\ \i1n$:
\be l_i(u)A(u)=\lambda_i(u)l_i(u). \ee

Multiplying \eqref{sys-original} with $l_i(u)\ \i1n$, we obtain the
characteristic form of \eqref{sys-original}:
\be\label{sys}
l_i(u)\Big(\pd ut+\lambda_i(u)\pd
ux\Big)=f_i(u):=l_i(u)B(u)\quad\i1n. \ee
Clearly,
\be\label{f_i} f_i(0)=0\quad \i1n. \ee

In what follows, we assume that there exist $l,m\in \mathbb{Z},
0\leq l\leq m\leq n$, such that on the domain under consideration
\be\label{eigenvalue}
\la_p(u)<\la_q(u)\equiv 0< \la_r(u) \quad (p=1,\cdots,l;\
q=l+1,\cdots m;\ r=m+1,\cdots,n). \ee

Let us assume that the initial condition is \be\label{initial}
u(0,x)=\vph(x),\quad x\in [0,L], \ee and the boundary conditions
take the following \emph{nonlocal} form:
\begin{align}\nonumber
&v_r(t,0)=G_r(t,v_1(t,0),\cdots,v_m(t,0),v_{l+1}(t,L),\cdots,v_n(t,L))+H_r(t)\\\label{BC-v_r-0}
&\hspace{82mm} (r=m+1,\cdots,n),\\ \nonumber
&v_p(t,L)=G_p(t,v_1(t,0),\cdots,v_m(t,0),v_{l+1}(t,L),\cdots,v_n(t,L))+H_p(t)\\\label{BC-v_p-L}
&\hspace{91mm}(p=1,\cdots,l),
\end{align}
where
\be\label{v_i} v_i=l_i(u)u\quad \i1n, \ee
$v_i$ being  called the \emph{diagonalized variables} corresponding
to $\lambda_i(u)$, $L$ is the length of the space interval,
$G_p,G_r,H_p,H_r\ (p=1,\cdots,l; r=m+1,\cdots,n)$ are all suitably
smooth functions. Without loss of generality, we assume that
\be\label{G_p-G_r} G_p(t,0,\cdots,0)\equiv G_r(t,0,\cdots,0)\equiv 0
\quad (p=1,\cdots,l;\ r=m+1,\cdots,n). \ee

The basic features of this kind of \emph{nonlocal} boundary
conditions can be described as follows: on the whole boundary ($x=0$
and $x=L$) of the domain under consideration, the diagonalized
variables $(v_{m+1}(t,0),\cdots, v_n(t,0),
v_1(t,L),\cdots,v_l(t,L))$ corresponding to the \emph{coming
characteristics} can be expressed explicitly by all the other
diagonalized variables $(v_1(t,0),\cdots,
v_m(t,0),v_{l+1}(t,L),\cdots, v_n(t,L))$. It is a generalization of
the \emph{local} nonlinear boundary conditions considered in
\cite{LiJin, Wang1}, however, the local existence and uniqueness of
$C^1$ solution to this mixed problem \eqref{sys} and
\eqref{initial}-\eqref{BC-v_p-L} can still be treated under the
framework of \cite{LiYubook}. In order to study the exponential
stabilization of the $H^2$ solution, Coron et al. \cite{Coron}
established the existence and uniqueness of $H^2$ solution to this
kind of  mixed problem under the assumption that there are no zero
eigenvalues.

In this paper, we first establish the theory on semiglobal $C^1$
solution to the mixed problem \eqref{sys} and
\eqref{initial}-\eqref{BC-v_p-L} in Section 2, then, under the
assumption that system \eqref{sys} possesses no zero eigenvalues, by
means of a constructive method, we obtain the results on the local
exact boundary controllability and observability in Section 3.
Direct applications to Saint-Venant system and 1-D quasilinear wave
equation are given in Sections 4 and 5, respectively. Finally, with
the linearized Saint-Venant system and the 1-D linear wave equation
as examples, we show  that the number of both boundary controls and
boundary observations can not be reduced, and consequently, we
conclude that the exact boundary controllability for a system in a
network with loop can not be realized generically.

\section{Semiglobal $C^1$ solution to the nonlocal mixed problem}

\begin{thm}{\bf (Semiglobal $C^1$ solution)}\label{semiglobal}
Suppose that on the domain under consideration,
$l_i,\la_i,f_i,G_p,G_r,H_p,H_r\ (i=1,\cdots,n;p=1,\cdots,l;
r=m+1,\cdots,n)$ and $\vph$ are all $C^1$ functions with respect to
their arguments. Suppose furthermore \eqref{f_i}-\eqref{eigenvalue}
and \eqref{G_p-G_r} hold and the conditions of $C^1$ compatibility
are satisfied at the points $(t,x)=(0,0)$ and $(0,L)$. For any
preassigned and possibly quite large $T>0$, if
$\|\vph\|_{C^1[0,L]}$, $\|H_p\|_{C^1[0,T]}\ (p=1,\cdots,l)$ and
$\|H_r\|_{C^1[0,T]}\ (r=m+1,\cdots,n)$ are sufficiently small
(depending on $T$), then the mixed problem \eqref{sys} and
\eqref{initial}-\eqref{BC-v_p-L} admits a unique semiglobal $C^1$
solution $u=u(t,x)$ with small $C^1$ norm on the domain
$R(T)=\{(t,x)|0\leq t\leq T, 0\leq x\leq L\}$. Moreover, when
$\pd{G_p}{t}(t,\cdot)\ (p=1,\cdots,l)$ and $\pd{G_r}{t}(t,\cdot)$
$(r=m+1,\cdots,n)$ satisfy local Lipschitz conditions with respect
to the variable $v=(v_1,\cdots,v_n)^T$, we have the following
estimate
 \be\label{C1-estimate} \|u\|_{C^1[R(T)]}\leq
C(\|\vph\|_{C^1[0,L]} +\sum_{p=1}^l\|H_l\|_{C^1[0,T]}+
\sum_{r=m+1}^n\|H_r\|_{C^1[0,T]}), \ee
 where $C$ is a positive constant possibly depending on $T$.
\end{thm}

\begin{proof}
Assume that $u=u(t,x)$ is a $C^1$ solution to the mixed problem
\eqref{sys} and \eqref{initial}-\eqref{BC-v_p-L} on $R(T)$. Let
 \be
 \wt u(t,x)=u(t,L-x),\quad (t,x)\in R(T).
 \ee
$\wt u=\wt u(t,x)$ satisfies  the following mixed problem on $R(T)$:
 \begin{align}
 & l_i(\wt u)\Big(\pd {\wt u}t-\lambda_i(\wt u)\pd {\wt u}x\Big)=f_i(\wt u)\quad\i1n,\\
 & \wt u(0,x)=\vph(L-x),\quad x\in [0,L],
  \\\nonumber
 &{\wt v}_r(t,L)=G_r(t,{\wt v}_1(t,L),\cdots,{\wt v}_m(t,L),{\wt v}_{l+1}(t,0),\cdots,{\wt v}_n(t,0))+H_r(t)\\
 &\hspace{82mm} (r=m+1,\cdots,n),
   \\ \nonumber
 &{\wt v}_p(t,0)=G_p(t,{\wt v}_1(t,L),\cdots,{\wt v}_m(t,L),{\wt v}_{l+1}(t,0),\cdots,{\wt v}_n(t,0))+H_p(t)\\
 &\hspace{91mm} (p=1,\cdots,l),
 \end{align}
where
\be
{\wt v}_i(t,x)=l_i(\wt u(t,x))\wt u(t,x)=v_i(t,L-x)\quad \i1n.
\ee

Furthermore, let
 \begin{align}
 & U=\left(
        \begin{array}{c}
          u \\
          \wt u \\
        \end{array}
      \right)
 \in \mathbb{R}^{2n},\\\label{Lambda-i}
 &\Lambda_i(U)=\lambda_i(u),\ \Lambda_{n+i}(U)=-\lambda_i(\wt u) \quad \i1n,\\
 & L_i(U)=(l_i(u),0,\cdots,0)\in \mathbb{R}_{2n},
   \ L_{n+i}(U)=(0,\cdots,0,l_i(\wt u))\in \mathbb{R}_{2n} \quad  \i1n,\\
 & F_i(U)=f_i(u),\ F_{n+i}(U)=f_i(\wt u)\quad  \i1n,
   \\\label{V-j}
 & V_j=L_j(U)U\quad (j=1,\cdots,2n).
 \end{align}
It is easy to see that
 \be V_i(t,x)=v_i(t,x),\ V_{n+i}(t,x)=\wt v_i(t,x)=v_i(t,L-x)
\quad \i1n, \ee
 and  $U(t,x)=\left(
             \begin{array}{c}
               u(t,x) \\
               \wt u(t,x) \\
             \end{array}
           \right)  $
is the $C^1$ solution to the mixed problem of the following
\emph{enlarged system} with \emph{local} boundary conditions  on
$R(T)$:
 \begin{align}\label{sys-U}
 & L_j(U)\Big(\pd Ut+\Lambda_j(U)\pd Ux\Big)=F_j(U)\quad (j=1,\cdots,2n),\\
 &U(0,x)=\left(
             \begin{array}{c}
               \vph(x) \\
               \vph(L-x) \\
             \end{array}
           \right),
\quad x\in [0,L],
  \\\nonumber
 &V_r(t,0)=G_r(t,V_1(t,0),\cdots,V_m(t,0),V_{n+l+1}(t,0),\cdots,V_{2n}(t,0))+H_r(t)\\
 &\hspace{88mm}  (r=m+1,\cdots,n),
   \\ \nonumber
 &V_{n+p}(t,0)=G_p(t,V_1(t,0),\cdots,V_m(t,0),V_{n+l+1}(t,0),\cdots,V_{2n}(t,0))+H_p(t),\\
 &\hspace{101mm} (p=1,\cdots,l),
  \\\nonumber
 &V_p(t,L)=G_p(t,V_{n+1}(t,L),\cdots,V_{n+m}(t,L),V_{l+1}(t,L),\cdots,V_{n}(t,L))+H_p(t)\\
 &\hspace{102mm} (p=1,\cdots,l),
  \\\nonumber
 &V_{n+r}(t,L)=G_r(t,V_{n+1}(t,L),\cdots,V_{n+m}(t,L),V_{l+1}(t,L),\cdots,V_{n}(t,L))+H_r(t)\\\label{BC-U}
 &\hspace{98mm} (r=m+1,\cdots,n).
 \end{align}

Since the boundary conditions in  the enlarged mixed problem are all
\emph{local}, the theory on the semiglobal classical solution in
\cite{LiJin} (or \cite{Wang1}) can be directly applied to show that
the mixed problem \eqref{sys-U}-\eqref{BC-U} admits a unique
semiglobal $C^1$ solution $U(t,x)=\left(
             \begin{array}{c}
               u(t,x) \\
               \wt u(t,x) \\
             \end{array}
           \right)  $
on $R(T)$. On the other hand, noting \eqref{Lambda-i}-\eqref{V-j},
it is easy to see that ${\wt U}(t,x)=\left(
             \begin{array}{c}
               \wt u(t,L-x) \\
               u(t,L-x) \\
             \end{array}
           \right)  $
is also a $C^1$ solution to  the same  mixed problem
\eqref{sys-U}-\eqref{BC-U} on $R(T)$. By the uniqueness of $C^1$
solution (cf. \cite{LiYubook}), $U(t,x)\equiv \wt U(t,x)$, then $\wt
u(t,x)\equiv u(t,L-x)$.

Thus, from the existence of the semiglobal $C^1$ solution $U=U(t,x)$
to the \emph{enlarged} mixed problem \eqref{sys-U}-\eqref{BC-U} on
$R(T)$, we get immediately the existence of the semiglobal $C^1$
solution $u=u(t,x)$ to the original \emph{nonlocal} mixed problem
\eqref{sys} and \eqref{initial}-\eqref{BC-v_p-L} on $R(T)$.

Moreover, when $\pd{G_p}{t}(t,\cdot)\ (p=1,\cdots,l)$ and
$\pd{G_r}{t}(t,\cdot)$ $(r=m+1,\cdots,n)$ satisfy local Lipschitz
conditions with respect to the variable $v=(v_1,\cdots,v_n)^T$, the
estimate \eqref{C1-estimate} can be obtained directly from the above
argument.

\end{proof}

\begin{rem} The basic idea of the proof of Theorem \ref{semiglobal}
comes from the treatment in \cite{Coron}.
\end{rem}

\begin{cor}
Suppose that on the domain under consideration, $l_i,\la_i,f_i\
\i1n$ and $\vph$ are all $C^1$ functions with respect to their
arguments, and \eqref{f_i}-\eqref{eigenvalue} hold. If
$\|\vph\|_{C^1[0,L]}$ is sufficiently small, then Cauchy problem
\eqref{sys} and \eqref{initial} admits a unique global $C^1$
solution $u=u(t,x)$ on the whole maximum determinate domain
$D=\{(t,x)|t\geq 0, x_1(t)\leq x\leq x_2(t)\}$ (Fig. 1), where
$x=x_1(t)$ and $x=x_2(t)$ are two curves defined as follows:
 \be
\begin{cases}
\displaystyle\f {dx_1}{dt}=\max_{r=m+1,\cdots,n} \lambda_r(u(t,x_1)),\\
t=0:\ x_1=0
\end{cases} \ee
 and
  \be \begin{cases}
\displaystyle\f {dx_2}{dt}=\min_{p=1,\cdots,l} \lambda_p(u(t,x_2)),\\
t=0:\ x_2=L,
\end{cases} \ee
 respectively (see \cite{LiYubook}). Moreover, we have the following estimate
 \be \|u\|_{C^1[D]}\leq C\|\vph\|_{C^1[0,L]}. \ee

\end{cor}

\vskip 5mm

%------------------------------------------------------------------------figure 1
\begin{center}
\scriptsize \setlength{\unitlength}{1mm}
\begin{picture}(50,45)
\linethickness{1pt} \put(5,3){\vector(1,0){48}}
\put(10,-5){\vector(0,1){50}} \put(45,-2){\line(0,1){42}}
%\multiput(45,-2)(0,2){20}{\line(0,1){1.5}}
\qbezier(25,18)(20,5)(10,3) \qbezier(25,18)(35,6)(45,3)
\put(5,4){$0$} \put(5,40){$t$} \put(13,10){$x_1(t)$}
\put(34,10){$x_2(t)$} \put(11,-1){$0$} \put(41,-1){$L$}
\put(50,-1){$x$} \put(24,6){$D$}
\end{picture}
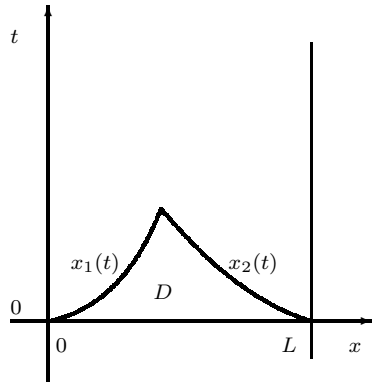
\vskip5mm  Figure 1. Maximum determinate domain $D$ of the Cauchy
problem
\end{center}

\vskip 5mm

\section{Local exact boundary controllability and observability}

When system \eqref{sys} possesses no zero eigenvalues (namely, $l=m$
in \eqref{eigenvalue}):
 \be\label{eigen-non-0} \la_r(u)<0<
\la_s(u) \quad (r=1,\cdots,m;s=m+1,\cdots,n), \ee
 the nonlocal boundary conditions \eqref{BC-v_r-0}-\eqref{BC-v_p-L} become
 \begin{align}\nonumber
&v_s(t,0)=G_s(t,v_1(t,0),\cdots,v_m(t,0),v_{m+1}(t,L),\cdots,v_n(t,L))+H_s(t)
 \\\label{BC-v_s-0}
&\hspace{80mm}\quad (s=m+1,\cdots,n),\\ \nonumber
&v_r(t,L)=G_r(t,v_1(t,0),\cdots,v_m(t,0),v_{m+1}(t,L),\cdots,v_n(t,L))+H_r(t)
 \\\label{BC-v_r-L}
&\hspace{87mm}\quad (r=1,\cdots,m),
\end{align}
where $v_i\ \i1n$ are still given by \eqref{v_i}, and without loss
of generality, we assume that \be\label{G_r-G_s}
G_r(t,0,\cdots,0)\equiv G_s(t,0,\cdots,0)\equiv 0 \quad
(r=1,\cdots,m;s=m+1,\cdots,n). \ee

Adopting the constructive method given in \cite{LiRao} to establish
the exact boundary controllability, we obtain

\begin{thm}{\bf (Exact boundary controllability)}\label{controllability}
Suppose that $l_i,\la_i,f_i,G_i\ (i=1,\cdots,n)$ and $\varphi$ are
all $C^1$ functions with respect to their arguments. Suppose
furthermore that \eqref{f_i},\eqref{eigen-non-0} and \eqref{G_r-G_s}
hold. Let \be\label{T} T>L \max_{i=1,\cdots,n}\frac{1}{|\la_i(0)|}.
\ee For any given initial data $\vph$ and final data $\psi$, if
$\|\varphi\|_{C^1[0,L]}$ and $\|\psi\|_{C^1[0,L]}$ are sufficiently
small, then these exist boundary controls $H_i(t)\ \i1n$ with small
$C^1[0,T]$ norms, such that the corresponding mixed problem
\eqref{sys},\eqref{initial} and \eqref{BC-v_s-0}-\eqref{BC-v_r-L}
admits a unique semiglobal $C^1$ solution $u=u(t,x)$ with small
$C^1$ norm on the domain $R(T)=\{(t,x)|0\leq t\leq T, 0\leq x\leq
L\}$, which satisfies exactly the final condition \be
u(T,x)=\psi(x),\quad x\in [0,L]. \ee
\end{thm}

Applying the constructive method given  in \cite{Li3} to establish
the exact boundary observability, we have

\begin{thm}{\bf (Exact boundary observability)}
\label{observability} Suppose that $l_i,\la_i,f_i,G_i,H_i$
$(i=1,\cdots,n)$ and $\vph$ are all $C^1$ functions with respect to
their arguments, and $\pd {G_i}{t}(t,\cdot)\ \i1n$ satisfy local
Lipschitz conditions with respect to the variable
$v=(v_1,\cdots,v_n)^T$. Suppose furthermore that
\eqref{f_i},\eqref{eigen-non-0} and \eqref{G_r-G_s}-\eqref{T} hold.
Suppose finally that $\|\vph\|_{C^1[0,L]}$ and $\|H_i\|_{C^1[0,T]}\
\i1n$ are sufficiently small, and the conditions of $C^1$
compatibility for the mixed problem \eqref{sys},\eqref{initial} and
\eqref{BC-v_s-0}-\eqref{BC-v_r-L} are satisfied at the points
$(t,x)=(0,0)$ and $(0,L)$. Then the initial data $\vph$ can be
uniquely determined by the boundary observations $\overline v_r(t):=
v_r(t,0) (r=1,\cdots,m)$ and $\overline{\overline v}_s(t):=
v_s(t,L)\ (s=m+1,\cdots,n)$ together with the known boundary
functions $H_i(t)\ \i1n$. Moreover, the following observability
estimate holds: \be \|\vph\|_{C^1[0,L]}\leq C
(\sum_{r=1}^m\|\overline v_r\|_{C^1[0,T]}
+\sum_{s=m+1}^n\|\overline{\overline
v}_s\|_{C^1[0,T]}+\sum_{i=1}^n\|H_i\|_{C^1[0,T]}), \ee where $C$ is
a positive constant possibly depending on $T$.
\end{thm}

\section{Application 1---Saint-Venant system}

Consider the Saint-Venant system for a horizontal and cylindrical
canal (see \cite{Halleux, Leugering, Li1, StVenant})
\be\label{Saint-Venant}
\begin{cases}
A_t+(AV)_x=0,\\
V_t+S_x=0,
\end{cases}
\ee where $A>0$ stands for the area of the cross section occupied by
the water, $V$ is the average velocity over the cross section and
\be S=\f 12 V^2+gH(A)+gY, \ee where $g$ is the gravity constant, $Y$
is the altitude of the canal bed (we may assume $Y=0$ without loss
of generality), $H$ is the depth of water,  which is a $C^1$
function of $A$ satisfying \be\label{H'(A)} H'(A)>0, \quad \forall
A>0. \ee

Let the initial condition be \be\label{SV-initial} A(0,x)=A_0(x),\
V(0,x)=V_0(x),\quad x\in [0,L], \ee and the boundary conditions take
the following nonlocal form:
\begin{align}\label{BC-S}
&S(t,0)-S(t,L)=h(t),\\\label{BC-Q}
&Q(t,0)-Q(t,L)=\overline h(t),
\end{align}
where $Q=AV$ denotes the flux.

We discuss system \eqref{Saint-Venant} near a  constant subcritical
equilibrium $(\wt A,\wt V)\ (\wt A>0)$ which satisfies
\be\label{subcritical} \wt V^2<g\wt AH'(\wt A). \ee

Introducing Riemann Invariants
 \be\label{sys-rs} r=\f 12(V-\wt
V-G(A)),\quad s=\f 12(V-\wt V+G(A)), \ee where \be G(A)=\int_{\wt
A}^A\sqrt{\f {gH'(A)}{A}}, \ee then
 \be V=r+s+\wt V,\quad
A=G^{-1}(s-r), \ee where $G^{-1}$ denotes the inverse function of
$G$. By \eqref{subcritical}, in a $C^1$ neighbourhood of $(A,V)=(\wt
A,\wt V)$ (correspondingly, $(r,s)=(0,0)$), \eqref{Saint-Venant} can
be equivalently rewritten as \be\label{SV-sys-rs}
\begin{cases}
r_t+\la_1\, r_x=0,\\
s_t+\la_2\, s_x=0,
\end{cases}
\ee
where
\be
\la_1=V-\sqrt{gAH'(A)}<0<\la_2=V+\sqrt{gAH'(A)}.
\ee

The initial condition \eqref{SV-initial} becomes
\be\label{SV-initial-rs} r(0,x)=r_0(x),\ s(0,x)=s_0(x),\quad x\in
[0,L], \ee where \be r_0(x)=\f 12(V_0(x)-\wt V-G(A_0(x))),\quad
s_0(x)=\f 12(V_0(x)-\wt V+G(A_0(x))). \ee

In order to change nonlocal boundary conditions
\eqref{BC-S}-\eqref{BC-Q} into the form of
\eqref{BC-v_r-0}-\eqref{BC-v_p-L}, we first rewrite them as
\begin{align}
&P_1:= \f 12(V_1^2-V_2^2)+g(H(A_1)-H(A_2))-h(t)=0,\\
&P_2:= A_1V_1-A_2V_2-\overline h(t)=0,
\end{align}
where
 \be V_1=V(t,0), V_2=V(t,L), A_1=A(t,0), A_2=A(t,L). \ee Let
 \be r_1=r(t,0), r_2=r(t,L), s_1=s(t,0), s_2=s(t,L). \ee Then, at the
point $(A,V)=(\wt A,\wt V)$ (correspondingly, $(r,s)=(0,0)$),
 \be
\det\left(\frac{\partial(P_1,P_2)}{\partial(s_1,r_2)}\right)
=2\sqrt{\frac{\wt A}{gH'(\wt A)}}\cdot(\wt V^2-g\wt A H'(\wt A))<0.
\ee
 By the Implicit Function Theorem, in a $C^1$ neighbourhood of
$(A,V)=(\wt A,\wt V)$ (correspondingly, $(r,s)=(0,0)$), boundary
conditions \eqref{BC-S}-\eqref{BC-Q} can be furthermore rewritten as
\begin{align}\label{SV-BC-rs-1}
&s(t,0)=F(t,r(t,0),s(t,L))+f(t),\\\label{SV-BC-rs-2}
&r(t,L)=\overline F(t,r(t,0),s(t,L))+\overline f(t),
\end{align}
where $F,\overline F$ are $C^1$ functions with respect to their
arguments, and, without loss of generality, we may assume that
 \be F(t,0,0) \equiv
\overline F(t,0,0) \equiv 0, \ee
 consequently,
  \be \|(h,\overline
h)\|_{(C^1[0,T])^2} \rightarrow 0 \Longleftrightarrow \|(f,\overline
f)\|_{(C^1[0,T])^2} \rightarrow 0. \ee

Applying Theorem \ref{semiglobal} to the mixed problem \eqref{SV-sys-rs},
\eqref{SV-initial-rs} and \eqref{SV-BC-rs-1}-\eqref{SV-BC-rs-2}, we obtain

\begin{thm}{\bf (Semiglobal $C^1$ solution)}\label{SV-semiglobal}
Let $(\wt A,\wt V)\ (\wt A>0)$ be a constant subcritical
equilibrium. For any preassigned and possibly quite large $T>0$, if
$\|(A_0-\wt A,V_0-\wt V)\|_{(C^1[0,L])^2}$
 and $\|(h,\overline h)\|_{(C^1[0,T])^2}$ are sufficiently small,
and the conditions of $C^1$ compatibility  are satisfied at the
points $(t,x)=(0,0)$ and $(0,L)$, then the mixed problem
\eqref{Saint-Venant} and \eqref{SV-initial}-\eqref{BC-Q} admits a
unique semiglobal $C^1$ solution $(A,V)=(A(t,x),V(t,x))$ on
$R(T)=\{(t,x)|0\leq t\leq T, 0\leq x\leq L\}$,  $\|(A-\wt A,V-\wt
V)\|_{(C^1[R(T)])^2}$ being  small, and the following estimate
holds:
  \be \|(A-\wt A,V-\wt V)\|_{(C^1[R(T)])^2}\leq C(\|(A_0-\wt
A,V_0-\wt V)\|_{(C^1[0,L])^2}+\|(h,\overline h)\|_{(C^1[0,T])^2}),
\ee where $C$ is a positive constant possibly depending on $T$.
\end{thm}

As in \cite{Li1}, by Theorem \ref{controllability} we get

\begin{thm}{\bf (Exact boundary controllability)}
\label{SV-controllability} Let $(\wt A,\wt V)\ (\wt A>0)$ be a
constant subcritical equilibrium. Let \be\label{SV-T} T>L \max
\left\{\f{1}{|\wt \lambda_1|},\f{1}{\wt \lambda_2}\right\}, \ee
where
 \be \wt \la_1=\wt V-\sqrt{g\wt AH'(\wt A)}<0< \wt \la_2=\wt
V+\sqrt{g\wt AH'(\wt A)}. \ee
 For any given initial data $(A_0,V_0)$ and final data $(A_T,V_T)$,
if $\|(A_0-\wt A,V_0-\wt V)\|_{(C^1[0,L])^2}$ and $\|(A_T-\wt
A,V_T-\wt V)\|_{(C^1[0,L])^2}$ are sufficiently small (possibly
depending on $T$), there exist boundary controls $(h(t),\overline
h(t))$ with small $\|(h,\overline h)\|_{(C^1[0,T])^2}$, such that
the mixed problem \eqref{Saint-Venant} and
\eqref{SV-initial}-\eqref{BC-Q} admits a unique semiglobal $C^1$
solution $(A,V)=(A(t,x),V(t,x))$ with small $\|(A-\wt A,V-\wt
V)\|_{(C^1[R(T)])^2}$ on $R(T)$, which satisfies exactly the final
condition: \be A(T,x)=A_T(x),V(T,x)=V_T(x),\quad x\in [0,L]. \ee
\end{thm}

As in \cite{GuLi}, by Theorem \ref{observability} we obtain

\begin{thm}{\bf (Exact boundary observability)}
\label{SV-observability} Let $(\wt A,\wt V)\ (\wt A>0)$ be a
constant subcritical equilibrium and  $T$ satisfy \e{SV-T}. If
$\|(A_0-\wt A,V_0-\wt V)\|_{(C^1[0,L])^2}$ and $\|(h,\overline
h)\|_{(C^1[0,T])^2}$ are sufficiently small, and the conditions of
$C^1$ compatibility  are satisfied at the points $(t,x)=(0,0)$ and
$(0,L)$, then the initial data $(A_0,V_0)$ can be uniquely
determined by the boundary observation $(\overline A(t),\overline
V(t)):= (A(t,0),V(t,0))$ together with the known boundary functions
$(h(t),\overline h(t))$. Moreover, the following observability
estimate holds: \be\label{SV-observ-estim} \|(A_0-\wt A,V_0-\wt
V)\|_{(C^1[0,L])^2} \leq C(\|(\overline A-\wt A, \overline V-\wt
V)\|_{(C^1[0,T])^2} +\|(h,\overline h)\|_{(C^1[0,T])^2}), \ee where
$C$ is a positive constant possibly depending on $T$.
\end{thm}

\begin{rem} Theorem \ref{SV-observability} still
holds if we take the boundary observations $(\overline
A(t),\overline V(t)):=(A(t,L),V(t,L))$ instead of $(A(t,0),V(t,0))$.
In fact, the  exact boundary observability can be realized as long
as the values $(A(t,0),V(t,0),A(t,L),V(t,L))$ or
$(r(t,0),s(t,0),r(t,L),s(t,L))$  can be uniquely determined from the
boundary observations together with boundary conditions
\eqref{BC-S}-\eqref{BC-Q}. For instance, if the boundary
observations are taken as $(\overline S(t),\overline
Q(t))=(S(t,0),Q(t,0))$ (or $(S(t,L),Q(t,L))$), the exact boundary
observability can be also realized with the following observability
estimate:
 \be \|(A_0-\wt A,V_0-\wt V)\|_{(C^1[0,L])^2} \leq
C(\|(\overline S-\wt S, \overline Q-\wt Q)\|_{(C^1[0,T])^2}
+\|(h,\overline h)\|_{(C^1[0,T])^2}), \ee
 where
 \be
\wt S=\f 12 \wt V^2+gH(\wt A),\quad  \wt Q=\wt A \wt V. \ee
\end{rem}

\begin{rem}
If the energy type boundary condition \eqref{BC-S} is replaced by
the  water level boundary condition
 \be H(A(t,0))-H(A(t,L))=h(t),\ee
 Theorems \ref{SV-semiglobal}-\ref{SV-observability} still hold.
\end{rem}

\section{Application 2---1-D quasilinear wave equation}

Consider the following 1-D quasilinear wave equation \be\label{wave}
u_{tt}-(K(u,u_x))_x=F(u,u_x,u_t), \ee where $K$ is a  $C^2$ function
with
 \be\label{K_v} K_v(u,v)>0 \ee
  and $F$ is a $C^1$ function with
 \be\label{F-0} F(0,0,0)=0. \ee

By \eqref{F-0}, $u\equiv 0$ is an equilibrium of \eqref{wave}. All
the discussions in this section will  be in a $C^1$ neighbourhood of
$(u,u_x,u_t)=(0,0,0)$.

Let the initial condition be
 \be\label{wave-initial}
u(0,x)=\vph(x),\ u_t(0,x)=\psi(x),\quad x\in [0,L] \ee
 and the boundary conditions take the following nonlocal form:
\begin{align}\label{Dirichlet}
&u(t,0)-u(t,L)=h(t),\\\label{Neumann}
&u_x(t,0)-u_x(t,L)=\overline h(t).
\end{align}
In particular, if $(h(t),\overline h(t))\equiv(0,0)$,
\eqref{Dirichlet}-\eqref{Neumann} become the usual periodic boundary
conditions.

Reducing the mixed problem \eqref{wave} and
\eqref{wave-initial}-\eqref{Neumann} to a quasilinear hyperbolic
system with boundary conditions in the form of
\eqref{BC-v_r-0}-\eqref{BC-v_p-L}, we will establish the theory of
the semiglobal $C^2$ solution and then the local exact boundary
controllability and observability.

Let \be v=u_x,\ w=u_t \ee and \be U=(u,v,w)^T. \ee \eqref{wave} can
be rewritten to the following first order quasilinear hyperbolic
system
 \be\label{wave-sys}
\begin{cases}
u_t=w,\\ v_t-w_x=0,\\w_t-K_v(u,v)\,v_x=\wt F(u,v,w):= F(u,v,w)+K_u(u,v)v
\end{cases}
\ee
 with
  \be \wt F(0,0,0)=0. \ee
   By \eqref{K_v}, \eqref{wave-sys} is a strictly hyperbolic system
with three distinct real eigenvalues
 \be
\lambda_1(U)=-\sqrt{K_v(u,v)}<\lambda_2(U)\equiv
0<\lambda_3(U)=\sqrt{K_v(u,v)} \ee
 and a complete set of left eigenvectors
  \be l_1(U)=(0,\sqrt{K_v(u,v)},1),\  l_2(U)=(1,0,0),\
l_3(U)=(0,-\sqrt{K_v(u,v)},1). \ee

The initial condition correspondingly becomes
\be\label{wave-initial-2} U(0,x)=(\vph(x),\vph'(x),\psi(x))^T\quad
x\in [0,L]. \ee

Let
 \be V_i=l_i(U)U\quad(i=1,2,3), \ee
 i.e., \be
V_1=\sqrt{K_v(u,v)}\,v+w,\  V_2=u,\  V_3=-\sqrt{K_v(u,v)}\,v+w. \ee
At the point $U=0$, we have
 \be \frac{\partial(V_1,V_2,V_3)}{\partial(u,v,w)} =\left(
   \begin{array}{ccc}
     0 & \sqrt{K_v(0,0)} & 1 \\
     1 & 0 & 0 \\
     0 & -\sqrt{K_v(0,0)}& 1 \\
   \end{array}
 \right),
\ee then
 \be\label{jacobi}
\frac{\partial(u,v,w)}{\partial(V_1,V_2,V_3)} =\left(
   \begin{array}{ccc}
     0 & 1 & 0\\
     \f 1{2\sqrt{K_v(0,0)}} & 0 & -\f 1{2\sqrt{K_v(0,0)}} \\
     \f 12 & 0& \f 12 \\
   \end{array}
 \right).
\ee

Noting the condition of $C^0$ compatibility  at the points
$(t,x)=(0,0)$ and $(0,L)$:
 \be \vph(0)-\vph(L)=h(0), \ee
  the boundary condition \eqref{Dirichlet} is equivalent to
\be\label{Dirichlet-2} w(t,0)-w(t,L)=h'(t). \ee

In order to reduce \eqref{Neumann} and \eqref{Dirichlet-2} into the
form of \eqref{BC-v_r-0}-\eqref{BC-v_p-L}, we first rewrite them to
\begin{align}
&P_1:= w(t,0)-w(t,L)-h'(t)=0,\\
&P_2:= v(t,0)-v(t,L)-\overline h(t)=0.
\end{align}
Let \be w_1=V_3(t,0),\ w_2=V_1(t,L). \ee At the point of $U=0$, by
\eqref{jacobi} it is easy to see that
  \be \det\left|\frac{\partial(P_1,P_2)}{\partial(w_1,w_2)}\right| =-\f
{1}{2\sqrt{K_v(0,0)}}<0, \ee
 then, in a $C^0$ neighbourhood of $U=0$, \eqref{Neumann} and
\eqref{Dirichlet-2} can be equivalently rewritten as
\begin{align}\label{wave-BC-0}
&V_3(t,0)=G_3(t,V_1(t,0),V_2(t,0),V_2(t,L),V_3(t,L))+H_3(t),\\\label{wave-BC-L}
&V_1(t,L)=G_1(t,V_1(t,0),V_2(t,0),V_2(t,L),V_3(t,L))+H_1(t),
\end{align}
where $G_1,G_3$ are $C^1$ functions with respect to their arguments
and  satisfy \be G_1(t,0,0,0,0)\equiv G_3(t,0,0,0,0)\equiv 0, \ee
consequently, \be \|(h',\overline h)\|_{(C^1[0,T])^2} \rightarrow 0
\Longleftrightarrow \|(H_1,H_3)\|_{(C^1[0,T])^2} \rightarrow 0. \ee

As in \cite{LiYu}(or \cite{Wang2}), applying Theorem
\ref{semiglobal} to the mixed problem \eqref{wave-sys},
\eqref{wave-initial-2} and \eqref{wave-BC-0}-\eqref{wave-BC-L}, we
obtain

\begin{thm}{\bf (Semiglobal $C^2$ solution)}\label{wave-semiglobal}
For any preassigned and possibly quite large $T>0$, if
$\|(\vph,\psi)\|_{C^2[0,L]\times C^1[0,L]}$ and $\|(h,\overline
h)\|_{C^2[0,T] \times C^1[0,T]}$ are sufficiently small (possibly
depending on $T$), and the conditions of $C^2$ compatibility  are
satisfied at the points $(t,x)=(0,0)$ and $(0,L)$, then the mixed
problem \eqref{wave} and \eqref{wave-initial}-\eqref{Neumann} admits
a unique semiglobal solution $C^2$ solution $u=u(t,x)$ with small
$C^2$ norm on the domain $R(T)=\{(t,x)|0\leq t\leq T, 0\leq x\leq
L\}$ and the following estimate holds: \be \|u\|_{C^2[R(T)]}\leq
C(\|(\vph,\psi)\|_{C^2[0,L]\times C^1[0,L]} +\|(h,\overline
h)\|_{C^2[0,T]\times C^1[0,T]}), \ee where $C$ is a positive
constant possibly depending on $T$.
\end{thm}

Based on Theorem \ref{wave-semiglobal}, adopting a similar
constructive method as in \cite{LiYu} (or \cite{Wang2}), we obtain
immediately

\begin{thm}{\bf (Exact boundary controllability)}
\label{wave-controllability} Let
 \be\label{wave-T}T>\f{L}{\sqrt{K_v(0,0)}}. \ee
  For any given initial data
$(\vph,\psi)$ and final data $(\Phi,\Psi)$, if the norms
$\|(\vph,\psi)\|_{C^2[0,L]\times C^1[0,L]}$ and
$\|(\Phi,\Psi)\|_{C^2[0,L] \times C^1[0,L]}$ are sufficiently small,
then there exist boundary controls $(h(t),\overline h(t))$ with
small  $\|(h,\overline h)\|_{C^2[0,T]\times C^1[0,T]}$, such that
the mixed problem \eqref{wave} and
\eqref{wave-initial}-\eqref{Neumann}admits a unique $C^2$ solution
$u=u(t,x)$ with small $C^2$ norm on $R(T)$, which satisfies exactly
the final condition \be u(T,x)=\Phi(x), u_t(T,x)=\Psi(x),\quad x\in
[0,L]. \ee
\end{thm}

By the constructive method in \cite{Li2} (or \cite{GuoWang}), we get

\begin{thm}{\bf (Exact boundary observability)}\label{wave-observability}
Let $T$ satisfy \eqref{wave-T}. If $\|(\vph,\psi)\|_{C^2[0,L]\times
C^1[0,L]}$ and $\|(h,\overline h)\|_{C^2[0,T]\times C^1[0,T]}$ are
sufficiently small, and the conditions of  $C^2$ compatibility  are
satisfied at the points $(t,x)=(0,0)$ and $(0,L)$, then the initial
data $(\vph,\psi)$ can be uniquely determined by the boundary
observations $(\overline u(t),\overline v(t)):= (u(t,0),u_x(t,0))$
together with the boundary functions $(h(t),\overline h(t))$.
Moreover, the following observability estimate holds:
\be\label{wave-observ-estim} \|(\vph,\psi)\|_{C^2[0,L]\times
C^1[0,L]}\leq C (\|(\overline u,\overline v)\|_{C^2[0,T]\times
C^1[0,T]} +\|(h,\overline h)\|_{C^2[0,T]\times C^1[0,T]}), \ee where
$C$ is a positive constant possibly depending on $T$.
\end{thm}

\begin{rem}
If the boundary observations $(\overline u(t),\overline v(t))$ are
taken as  $(u(t,0),u_x(t,L))$ or $(u(t,L),u_x(t,L))$ or
$(u(t,L),u_x(t,0))$ instead of $(u(t,0),u_x(t,0))$, Theorem
\ref{wave-observability} still holds. In fact, the exact boundary
observability always holds if $(u(t,0),u_x(t,0),$ $u(t,L),u_x(t,L))$
can be uniquely determined by the boundary observations and boundary
conditions \eqref{Dirichlet}-\eqref{Neumann}.
\end{rem}

\section{Exact boundary controllability for a system
in a network with loop can not be realized generically}

%%%%%%%%%%%%%%%%%%%%%%%%%%%%%%%%%%%%%%%%%%%%%%%%%%%%%%%%%%%%

In this section we give some examples to show that, generically
speaking, the number of both boundary controls and boundary
observations can not be reduced and then the exact boundary
controllability for a hyperbolic system in a network with loop can
not be realized.

{\bf 6.1. Linearized Saint-Venant system}

For the linearized Saint-Venant system near a  constant subcritical
equilibrium $(\widetilde A, \widetilde V)\ (\widetilde A>0)$
 \be{\partial \over \partial t}\left(\begin{array}{c}
    A\\
    V
    \end{array}\right) +
\left( \begin{array}{cc}
          \widetilde V&  \widetilde A \\
          gH'(\widetilde A) &   \widetilde V
        \end{array}
      \right)
 {\partial \over \partial x}
 \left(\begin{array}{c}
    A\\
    V
    \end{array}\right)=0,\ee
we consider the following nonlocal boundary conditions (cf. (4.31)
and (4.6)):
  \be A(t, L)- A(t, 0) =0\ee
  and
  \be V(t, L)- V(t, 0) =h(t),\ee
  which correspond to a loop.

  The two eigenvalues  and the corresponding left eigenvectors are given by
 \be \lambda_1= \widetilde V - \sqrt{g\widetilde A H'(\widetilde A)}
 <0< \lambda_2= \widetilde V + \sqrt{ g\widetilde A H'(\widetilde
 A)}\ee
 and
    \be l_1 = \big(\sqrt{g\widetilde AH'(\widetilde A)}, \ -\widetilde
A\big),\quad l_2 = \big(\sqrt{g\widetilde AH'(\widetilde A)},\
\widetilde A\big),\ee
   respectively. Using the Riemann invariants
\be \left(\begin{array}{c}
    r\\
    s
    \end{array}\right)=
\left( \begin{array}{cc}
          \sqrt{g\widetilde A H'(\widetilde A)}&  -\widetilde A \\
          \sqrt{g\widetilde A H'(\widetilde A)} &   \widetilde A
        \end{array}
      \right)
 \left(\begin{array}{c}
    A\\
    V
    \end{array}\right),\ee
   system (6.1) can be rewritten into the following diagonal
   form
\be \begin{cases}\displaystyle{\partial r\over \partial t}+
\lambda_1 {\partial r\over \partial x} =0,\\
 \displaystyle{\partial s\over \partial t} +\lambda_2{\partial s\over
\partial x}=0,\end{cases}\ee
 and (6.2)-(6.3)  are equivalently transformed into the following
 boundary conditions:
\be r(t, L)- r(t, 0) =-\widetilde A h(t)\ee
    and
    \be s(t, L)- s(t, 0) =\widetilde Ah(t).\ee
  For the control problem, there are formally two controls in (6.8)-(6.9), but they
are not independent.  We will show that system (6.7)-(6.9) is not
exactly controllable by means of $h(t)$.

 Let $(r_0, s_0)$ be a constant initial data satisfying
    \be  r_0 +  s_0 >0.\ee
    It is easy to see that the conditions of $C^1$ compatibility are
    satisfied at the point $(t,x)=(0,0)$ and $(0,L)$.
  Assume that there exists a control $h\in C^1[0,T]$, such
that system (6.7)-(6.9) with the initial data $(r_0,s_0)$ admits a
unique  $C^1$ solution $(r, s)=(r(t,x),s(t,x))$ on the domain
$R(T)=\{(t,x)|\, 0\leq t\leq T,\ 0\leq x\leq L\}$, which satisfies
the final conditions
  \be r(T, x) =  s(T, x) = 0,\quad 0\leq x\leq L.\ee
  Then, integrating (6.7) on $R(T)$ yields
  \be \begin{cases} r_0 L + \lambda_1\widetilde A\int_0^T h(t)dt=0,\\
   s_0 L - \lambda_2\widetilde A\int_0^T h(t)dt=0, \end{cases}\ee
  hence
  \be \lambda_2 r_0 +\lambda_1 s_0 =0.\ee
  Specially taking
  \be (r_0, s_0)= (\alpha\lambda_2, \lambda_1),\ee
 where $\alpha$ is a positive constant such that
  \be r_0+s_0 = \alpha \lambda_2 + \lambda_1 >0 \Longleftrightarrow \alpha
> {\sqrt{g\widetilde A H'(\widetilde A)} -\widetilde
V\over\sqrt{g\widetilde A H'(\widetilde A)}+\widetilde V },\ee
  we get a contradiction
 \be \lambda_1^2 + \alpha\lambda_2^2=0.\ee

%%%%%%%%%%%%%%%%%%%%%%%%%%%%%%%%%%%%%%%%%%%%%%%%%%%%%%%

{\bf 6.2. 1-D linear wave equation}

 First  we show that the number of boundary observations in Theorem
\ref{wave-observability} can not be reduced. For this purpose,
consider the following mixed problem for the linear wave equation
with the periodic boundary conditions:
\begin{empheq}
[left=\empheqlbrace, right=]{align}\label{linear-wave}
&\phi_{tt}-\phi_{xx}=0,\\
&\phi(t,0)=\phi(t,2 \pi),\\
&\phi_x(t,0)=\phi_x(t,2\pi),\\\label{linear-wave-IC}
&\phi(0,x)=\phi_0(x),\ \phi_t(0,x)=\phi_1(x),\quad  x\in[0,2\pi].
\end{empheq}

By Theorem \ref{wave-observability}, if the boundary observations
are chosen as $(\phi(t,0),\phi_x(t,0))$ and $T\geq 2\pi$, the exact
boundary  observability for
\eqref{linear-wave}-\eqref{linear-wave-IC} holds on the time
interval $[0,T]$. However, if the boundary observation is only
$\phi(t,0)$ (resp., $\phi_x(t,0)$), the exact boundary observability
for \eqref{linear-wave}-\eqref{linear-wave-IC} can not be realized
on any time interval $[0,T]\ (T>0)$. To show this, it suffices to
find a nontrivial solution to
\eqref{linear-wave}-\eqref{linear-wave-IC}, such that the boundary
observation $\phi(t,0)$ (resp., $\phi_x(t,0)$) is identically equal
to zero, while the initial data $(\phi_0(x),\phi_1(x))$ is not
identically zero. In fact,
 \be\label{sinnt-sinnx}
\phi(t,x)=\sin{nt}\sin{nx},\quad n\in \mathbb{Z}^+ \ee
 satisfies \eqref{linear-wave}-\eqref{linear-wave-IC} with
$(\phi_0(x),\phi_1(x))\equiv (0, n\sin{nx})$ and $\phi(t,0)\equiv
0$. Therefore, observing only $\phi(t,0)$ is not sufficient to
guarantee the exact boundary observability. Similarly,
 \be\label{cosnt-cosnx} \phi(t,x)=\cos{nt}\cos{nx},\quad n\in
\mathbb{Z}^+ \ee
 satisfies \eqref{linear-wave}-\eqref{linear-wave-IC} with
$(\phi_0(x),\phi_1(x))\equiv (\cos{nx},0)$ and $\phi_x(t,0)\equiv
0$. Then, observing only  $\phi_x(t,0)$ is not sufficient to
guarantee the exact boundary observability, either.

We now show that the number of boundary controls in Theorem
\ref{wave-controllability} can not be reduced. For this purpose, we
first  suppose that there exist $T>0$ and a boundary control
$\widetilde h(t)$ such that the solution  $y=y(t,x)$ of  the
following control system
\begin{empheq}
[left=\empheqlbrace, right=]{align}\label{Neumann-control}
&y_{tt}-y_{xx}=0,\\\label{Dirichlet-0}
&y(t,0)=y(t,2 \pi),\\
&y_x(t,0)=y_x(t,2\pi)+\widetilde h(t),\\\label{Neumann-control-IC}
&y(0,x)=y_0(x),\ y_t(0,x)=y_1(x),\quad  x\in[0,2\pi]
\end{empheq}
satisfies exactly the final null condition
 \be\label{final-zero}
y(T,x)\equiv y_t(T,x)\equiv 0, \quad x\in [0,2\pi]. \ee
Multiplying the wave equation \eqref{Neumann-control} by the
solution $\phi=\phi(t,x)$ to system
\eqref{linear-wave}-\eqref{linear-wave-IC}, and then integrating on
$[0,T]\times[0,2\pi]$, we obtain
 \be \int_0^T
\int_0^{2\pi}y_{tt}(t,x)\phi(t,x) dxdt =\int_0^T
\int_0^{2\pi}y_{xx}(t,x)\phi(t,x)dxdt. \ee
By integration by parts and using
\eqref{linear-wave}-\eqref{linear-wave-IC} and
\eqref{Dirichlet-0}-\eqref{final-zero}, it follows that
\be\label{integration by part}
\int_0^{2\pi}(-y_1(x)\phi_0(x)+y_0(x)\phi_1(x)) dx =-\int_0^T
\widetilde h(t)\phi(t,2\pi) dt. \ee

In particular, taking the initial data in \eqref{Neumann-control-IC}
to be
 \be y_0(x)=\sin{nx},\ y_1(x)\equiv 0, \quad x\in [0,2\pi] \ee
and $\phi(t,x)$ to be given by \eqref{sinnt-sinnx}, from
\eqref{integration by part} we get a contradiction  that
 \be n\int_0^{2\pi} \sin^2{nx} dx=0. \ee

\begin{rem}
Noting \eqref{Dirichlet-0}, we conclude from the above that: the
exact boundary controllability for a system in a network with loop
can not be realized generically.
\end{rem}

Similarly, it can be shown that if the initial data in
\eqref{Neumann-control-IC} is taken as
 \be y_0(x)\equiv 0, \
y_1(x)=\cos{nx}, \quad x\in [0,2\pi], \ee
 there do not exsit $T>0$ and a boundary control $h(t)$ such that
the solution $y=y(t,x)$ to the following control system
\be\label{Dirichlet-control}
\begin{cases}
y_{tt}-y_{xx}=0,\\
y(t,0)=y(t,2 \pi)+h(t),\\
y_x(t,0)=y_x(t,2\pi),\\
y(0,x)=y_0(x),\ y_t(0,x)=y_1(x),\quad  x\in[0,2\pi]
\end{cases}
\ee
 satisfies exactly the null final condition \eqref{final-zero}.

%\section*{ÖÂл}


\begin{thebibliography}{00}

%\bibitem{CoronBook}  J.-M. Coron,
%   Control and Nonlinearity,
%   Mathematical Surveys and Monographs {\bf 136}, American Mathematical Society, Providence, RI, 2007.

\bibitem{Coron}
 J.-M. Coron, G. Bastin and B. d'Andr¨¦a-Novel,
 \emph{Dissipative boundary conditions for one dimensional nonlinear hyperbolic systems},
  SIAM J. Control Optim., 47(2008), 1460-1498.

\bibitem{GuLi}Q. L. Gu, T. T. Li, \emph{Exact boundary observability of unsteady supercritical
   flows in a tree-like network of open canals}, Math. Methods Appl. Sci., 32(2008), 395-418.

\bibitem{GuoWang}L. N. Guo, Z. Q. Wang, \emph{Exact boundary observability for nonautonomous
   quasilinear hyperbolic systems}, Math. Methods Appl. Sci., 31(2008), 1956-1971.

\bibitem{Halleux}J. de Halleux, C. Prieur, J.-M. Coron, B. d'Andr\'{e}a-Novel, G. Bastin,
   \emph{Boundary feedback control in networks of open channels}, Automatica,  39(2003), 1365-1376.

\bibitem{Leugering}G. Leugering, E. G. Schmidt, \emph{On the modelling and stabilization of flows
   in networks of open canals}, SIAM J. Control Optim., 41(2002), 164-180.


\bibitem{Li1}T. T. Li, \emph{Exact boundary controllability of unsteady flows in a network of open
   canals},  Math. Nachr., 278(2005), 278-289.

\bibitem{Li2}T. T. Li, \emph{Exact boundary observability for 1-D quasilinear wave
   equations},  Math. Methods Appl. Sci., 29(2006), 1543-1553.

\bibitem{Li3}T. T. Li, \emph{Exact boundary observability for quasilinear hyperbolic
   systems},  ESAIM: COCV, 14(2008), 759-766.

\bibitem{LiJin}T. T. Li, Y. Jin,  \emph{Semi-global $C^1$ solution to the mixed
   initial-boundary value problem for quasilinear hyperbolic systems},
   Chinese Ann. Math. Ser. B,   22(2001),  325-336.

\bibitem{LiRao}T. T. Li, B. P. Rao,  \emph{Exact boundary controllability for
  quasilinear hyperbolic systems},  SIAM J. Control Optim.,  41(2003),  1748-1755.

\bibitem{LiYu}T. T. Li, L. X. Yu,  \emph{Exact boundary controllability for 1-D quasilinear
   wave equations},  SIAM J. Control Optim., 45(2006), 1074-1083.

\bibitem{LiYubook}T. T. Li, W. C. Yu, Boundary Value Problems for Quasilinear Hyperbolic Systems,
   Duke University Mathematics Series V, 1985.

\bibitem{StVenant}B. de Saint-Venant, \emph{Th\'{e}orie du mouvement
   non-permanent des eaux, avec application aux crues des
   rivi\`{e}res et \`{a} l'introduction des mar\'{e}es dans leur lit},
   C. R. Acad. Sci. Paris, 73(1871), 147-154, 237-240.

\bibitem{Wang1}Z. Q. Wang, \emph{Exact controllability for nonautonomous first order quasilinear
   hyperbolic systems},  Chinese Ann. Math. Ser. B, 27(2006), 643-656.

\bibitem{Wang2} Z. Q. Wang, \emph{Exact controllability for non-autonomous
   quasilinear wave equations}, Math. Methods Appl. Sci., 30(2007), 1311-1327.


\end{thebibliography}
\end{document}